\documentclass[a4paper, 12pt]{article}

\usepackage{amsmath,amsfonts,amssymb,a4,enumerate,epsfig,theorem,psfrag}
\usepackage[latin1]{inputenc}

\newcommand{\cC}{{\mathcal{C}}}
\newcommand{\cD}{{\mathcal{D}}}

\newcommand{\cS}{{\mathcal{S}}}
\newcommand{\Pp}{{\cal{P}}}

\newcommand{\bs}{{\bigskip}}
\newcommand{\noi}{{\noindent}}

\newcommand{\RR}{{\mathbb{R}}}

\newcommand{\CC}{{\mathbb{C}}}
\newcommand{\KK}{{\mathbb{K}}}

\newcommand{\Oo}{{\cal{O}}}

\newcommand{\Qq}{{\cal{Q}}}

\newcommand{\Uu}{{\cal{U}}}

\newcommand{\bum}{{\bf 1}}

\def\qed{\unskip\nobreak\hfil\penalty50\hskip1.75em\null\nobreak\hfil
$\blacksquare$ {\parfillskip=0pt \finalhyphendemerits=0 \par}\medbreak}

\newcommand\diag{\operatorname{diag}}
\newcommand\abs{\operatorname{abs}}

\newcommand\interior{\operatorname{int}}
\newcommand\sgn{\operatorname{sgn}}
\newcommand\tr{\operatorname{tr}}

\newtheorem{lemma}{Lemma}[section]

\newtheorem{theo}[lemma]{Theorem}
\newtheorem{prop}[lemma]{Proposition}
\newtheorem{proposition}[lemma]{Proposition}
\newtheorem{coro}[lemma]{Corollary}

\title{A geometric approach to spectrum interlacing}

\author{Ricardo S. Leite\thanks{Departamento de Matemática, UFES
		(ricardo.leite@ufes.br)}
	\and Carlos Tomei\thanks{Departamento de Matemática, PUC-Rio
		(carlos.tomei@gmail.com)}}

\begin{document}
\maketitle

\begin{abstract}
  We provide a detailed description of the maps associated with spectral
  interlacing in two scenarios, for rank one perturbations and bordering of symmetric and
  Hermitian matrices. The arguments rely on standard techniques of nonlinear
  analysis.
\end{abstract}

\medbreak


\smallbreak

{\noindent\bf MSC-class:} 15A29, 15A42, 15B57

\section{Introduction}

We recall two standard results, presented in more detail in \cite{HornJohnson}. Endow $\CC^n$ with the Euclidean inner product $\langle \cdot, \cdot \rangle$  which is anti-linear in the second coordinate. For $v \in \CC^n$, the linear rank one map $v \otimes v = v v^\ast $ is  $u \mapsto \langle u, v \rangle \ v$. Let $S$ be an $n \times n $ Hermitian  matrix, with  ordered eigenvalues $\lambda_1 \le \ldots \le \lambda_n$.

\begin{theo} [Cor. 4.3.9, Th. 4.3.21 \cite{HornJohnson}] \label{standard}  For $v \in \CC^n$, let the eigenvalues of $T = T(v) = S + v \otimes v $ be $\mu_1 \le \ldots \le \mu_n$. Then the eigenvalues of $S$ and $T$ interlace,
\[ \lambda_1 \le \mu_1 \le \lambda_2 \le \mu_2 \le \ldots\le \lambda_n \le \mu_n \ . \]
Conversely, for a sequence $\{\mu_j\}$ interlacing $\{\lambda_k\}$ as above, there is $v \in \CC^n$ for which the eigenvalues of $T= S + v \otimes v$ are $\{\mu_j\}$.
\end{theo}

The second result is Cauchy's interlacing theorem.

\begin{theo} [Th. 4.3.17 \cite{HornJohnson}] \label{Cauchy} Let $v \in \CC^n$, $c \in \RR$, $\mu_1 \le \mu_2 \le \ldots \le \mu_{n+1}$ be the eigenvalues of the bordered matrix
\[
T = T(v,c)=\begin{pmatrix}
	S & v^\ast \\
	v & c
\end{pmatrix}.
\]
Then the eigenvalues of $S$ and $T$ interlace,
\[ \mu_1 \le \lambda_1 \le \mu_2 \le \lambda_2 \le \mu_3 \le \ldots \le \lambda_n \le \mu_{n+1} \ . \]
For a sequence $\{\mu_j\}$ interlacing $\{\lambda_k\}$, there exist $v \in \CC^n$ and $c >0$  for which the eigenvalues of $T = T(v,c)$ are $\{\mu_j\}$.
\end{theo}

The results can be found in essentially any advanced book on linear algebra. Extensions abound. Ionascu \cite{Ionascu} indicates a number of references (from which we emphasize \cite{Vasudeva}) proving related results for compact self-adjoint operators $S$ on a separable Hilbert space. Simon \cite{Simon} considers finer aspects of the spectrum of rank one perturbations of (mostly) Schrödinger operators with very interesting applications, which are beyond the scope of this text.

\medskip
In this text,  we cast these results in geometric terms.
Fix  a normalized eigenbasis $\Qq = [q_1, \ldots, q_n]$  of $S$, arranged as columns of a unitary matrix $Q$.
Let  $\Oo_Q \in \CC^n$, the \textit{positive orthant associated with $\Qq$}, be
the set of  vectors of the form $Q p$, where $p \in \RR^n$ has nonnegative entries.

A vector $v \in \RR^n$ is {\it ordered} if $v_1 \le \ldots \le v_n$.
Write $\sigma_o(S) \in \RR^n$ for the ordered vector with entries given by the eigenvalues of $S$. For $r > 0$, define the polytopes
\[ \Pp_F= [\lambda_1,\lambda_2] \times [\lambda_2,\lambda_3] \times \ldots
\times [\lambda_n, \infty)  \ , \]
\[  \Pp_G= (-\infty, \lambda_1] \times [\lambda_1,\lambda_2] \times
[\lambda_2,\lambda_3] \times \ldots \times [\lambda_n, \infty) \ , \]
two half-open boxes, and the sphere
$ \cS(r) =  \{ v \in \CC^n, \| v \| = r \} $.
\begin{theo} \label{teorema} Let $S$ be an $n \times n $ Hermitian matrix, with  spectrum $\lambda_1 < \ldots < \lambda_n$, an eigenbasis $\Qq$ and positive orthant $\Oo_Q$. The following maps are homeomorphisms, and diffeomorphisms between the interior of their domain and image.
\begin{align*}F : \cD_F = \Oo_Q  &\to \ \Pp_F\hspace*{1.1cm} , \hspace*{0.8cm}
	G :  \cD_G \ = \ \Oo_Q \times \RR \to \ \Pp_G\\
	v \  \mapsto \ &\sigma_o (S +  v \otimes v)
	\hspace*{1.8cm} ( v , c ) \  \mapsto \sigma_o (T(v,c))
\end{align*}
and their restrictions
\[ F^r: \cD_F^r = \Oo_Q \cap \cS(r) \to \Pp_F^r = \Pp_F\cap \{ \mu \in \RR^n, \
\ \sum_j \mu_j \ = \ r^2 + \sum_j \lambda_j  \}\ , \]
\[ \quad G^c:  \Oo_Q \times \{c\} \to  \Pp_
G^c = \Pp_G\cap \{ \mu \in \RR^{n+1},\  \sum_j \mu_j \ = \  \sum_k \lambda_k + c \}  \ , \]
\[ \quad G^{r,c}: (\Oo_Q \cap (\cS(r)) \times \{c\} \to  \Pp_G^c \cap \{ \mu \in \RR^{n+1},\  \sum_j \mu_j^2 \ = \  \sum_k \lambda_k^2 + 2 r^2 + c^2 \} . \]

\end{theo}

If $S$ has distinct eigenvalues, Theorems \ref{standard} and \ref{Cauchy} then follow. Full generality is attained by taking limits.

The general rank one
Hermitian perturbation matrix is of the form $c v \otimes v $ for a real unit
vector $v$ and $c \in \RR$. The sign of $c$ specifies if the perturbation of $S$
pushes the spectrum to the right (the case $c >0$) or to the left ($c <0$). For the results above,  $c \ge 0$: minor alterations
handle $c \le 0$. Clearly, the interlacing property is associated with the geometry of the polytopes $\Pp_F$ and $\Pp_G$.

It is rather intriguing that the interior of an orthant $\Oo_Q$ is taken by $F$ to $\Pp_F$, a closed box with a face removed. As we shall see in the proof, some faces of $\Oo_Q$ are {\it creased} by $F$, giving rise to two faces of $\Pp_Q$. Something similar happens with $G$, but now $\Pp_G$ is a box with two faces  removed.

This is what happens for $n=2$. The horizontal axis is taken to the union of a horizontal and a vertical segment. The vertical axis is sent to itself.

\begin{figure} [h]
  \begin{center}
    \includegraphics[width=350pt]{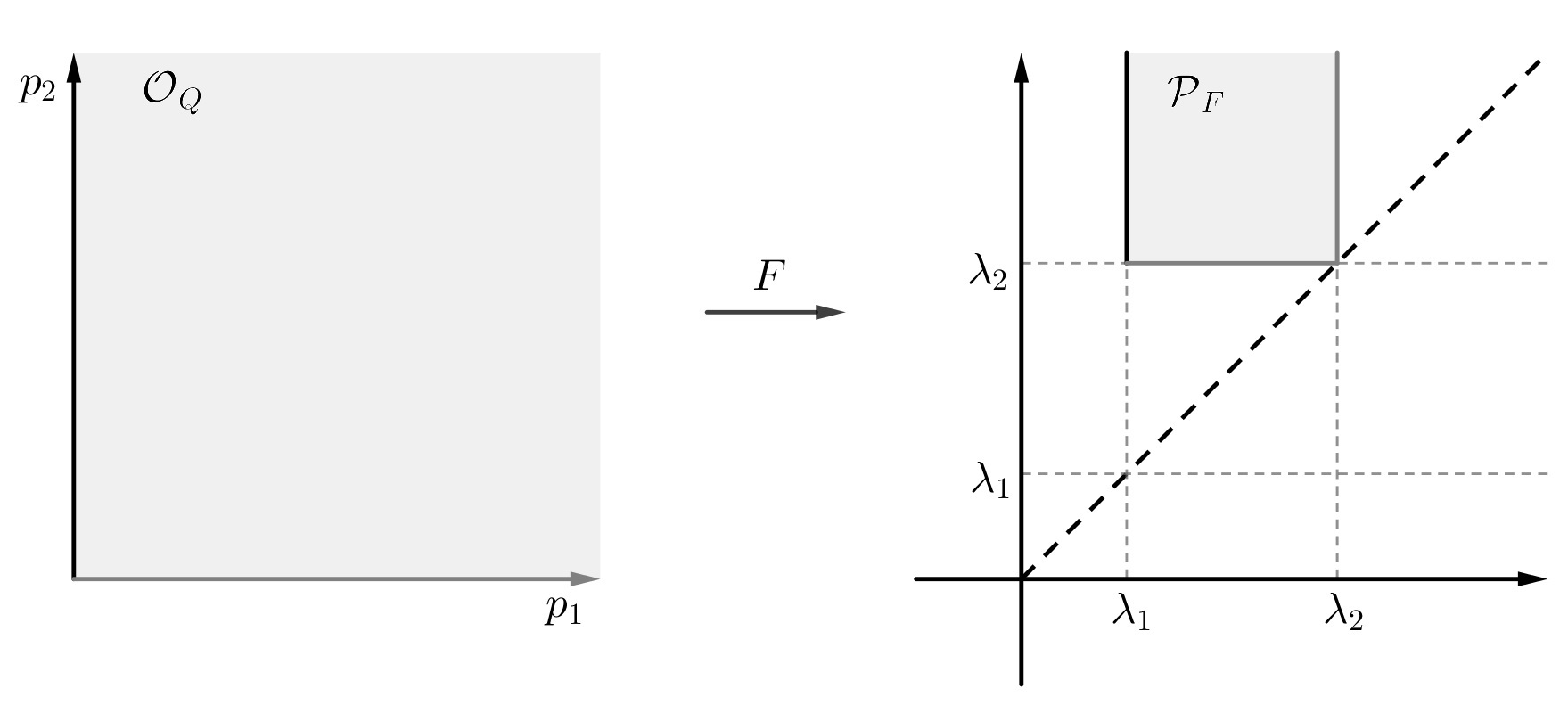}
    \caption{$F: \cD \to \RR^2$}
    \label{figura}
  \end{center}
\end{figure}

\medskip
The simple geometry of the maps $F$ and $G$ has implications to the computation of their inverses, frequently described as an inverse problem (\cite{HornJohnson}). For $F$, given a symmetric matrix $S$ with ordered, simple, spectrum $\lambda$ and an interlacing ordered
$n$-tuple $\mu$,
we look for a rank one perturbation $c v \otimes v$ such that the spectrum of
$ S + c v \otimes v$ is $\mu$.
Theorem \ref{teorema} shows that, in principle, the problem is solvable by numerical continuation starting from any interior point  of $\cD_F$, as
there are no critical values there. The same argument proves that continuation from an interior point of $\cD_G$ obtains the inverse of $G$.

\bigskip

Given a function $f: X \to Y$, the {\it preimages} of $y \in Y$ are the points in the set $f^{-1}(y)= \{ x \in X , \ f(x) = y\}$.
We now consider the preimages of the maps in the previous theorem. We have to distinguish between matrices with real or complex entries. Let $\KK = \RR$ or $\CC$ and define
\[ \abs: \KK^n \to \Oo_Q , \quad v = \sum_{j=1}^n c_j\ q_j \ \mapsto \
\sum_{j=1}^n |c_j| \ q_j \ . \]

\begin{theo} \label{teoremadois} Let $S$ as in the previous theorem.
\begin{enumerate}
\item   Say the entries of $S$ lie in $\KK = \RR$ or $\CC$, $v \in \KK^n$. Then $F$ and $G$ extend to
\begin{align*} \hat F : \KK^n  &\to \ \Pp_F\hspace*{1.0cm} , \hspace*{0.8cm}
	\hat G :  \KK^n \times \RR \to \ \Pp_G \hspace*{1.3cm} . \\
	v \  \mapsto \ &\sigma_o (S +  v \otimes v)
	\hspace*{1.8cm} ( v , c ) \  \mapsto \sigma_o (T(v,c))
\end{align*}
Moreover, $\hat  F(v)= \hat  F( w) \Leftrightarrow \hat  G(v)= \hat  G( w) \Leftrightarrow \abs(v) = \abs(w)$. In particular, all preimages of a point belong to the same sphere $\cS(r)$.
\item If $\KK = \RR$, a point $\mu \in \Pp_F$ belonging to exactly $k$ faces has $2^{n-k}$ preimages under $\hat  F$ or $\hat  G$. If $\KK = \CC$, for both functions the preimages of $\mu$ form a product of $n-k-1$ circles: a torus.
\end{enumerate}
\end{theo}

Recently, Maciazek and Smilansky \cite{MS} considered analogous inverse problems and pointed out the relevance of discrete information provided by strings of signs. We believe our presentation sheds some light on the issue.

\medskip

Theorems \ref{teorema} and \ref{teoremadois} are a strengthened version of a {\it very} special case of the celebrated Horn's conjecture \cite{Horn2}, whose resolution, after work by several authors (\cite{Kl}, \cite{HR}, \cite{KT2}, \cite{KTW}),  is beautifully described in \cite{KT}. The conjecture answers a question by Weyl \cite{W}: what are the possible spectra of the sum $A+B$ of two Hermitian matrices of given spectrum? Horn originally provided a list of linear inequalities on the eigenvalues of the three matrices which provide necessary and sufficient conditions relating their spectra. 
For $A = S$, $B = v \otimes v$ with $\| v \| = r$, Horn's conjecture states that the image of the map $F^r$ is indeed $\Pp_F^r$.

Part of the statements in Theorems \ref{teorema} and \ref{teoremadois}
may be deduced from the sophisticated machinery of
symplectic geometry. To give an idea of a more familiar context, the Schur-Horn theorem for Hermitian
matrices \cite{Horn1} is a consequence of the powerful theorems about the convexity of the
image of moment maps of torus actions by Atiyah (\cite{Atiyah}) and
Guillemin-Sternberg (\cite{GS}). The result for symmetric matrices then follows  by an
argument by Duistermaat (\cite{Duistermaat}). Similarly, the surjectivity of the maps $F, F^r$ and $G$  also follow from convexity arguments, once the appropriate symplectic setting is identified. Here, we take what
Thompson (\cite{Thompson})  calls a low road in linear algebra, but gain some
information which does not follow directly from rote application of these more
general results.

\bigskip

The proof of Theorem \ref{teorema} relies on a combination of well known facts of real analysis, condensed in Lemma \ref{lemma:topo}. The verification of the hypotheses of the lemma is somewhat different for $F$ and $G$. In both cases, the theorem is proved by induction on the dimension. In the inductive step, we
see how faces of the domain are `creased' by either $F$ or $G$ so as to obtain
the faces of the image parallelotope.
Theorem \ref{teoremadois} is a simple consequence of Theorem \ref{teorema}.

\bigskip
The authors are supported by CNPq, CAPES, and FAPERJ.
They are also grateful to an anonymous reader of a previous version of this text, who indicated errors and suggested a number of improvements.

\section{A real analysis lemma} \label{provateorema}

The  outline of the proof of Theorem \ref{teorema} is the same for the functions $F: \cD_F \to \Pp_F$ and $G: \cD_G \to \Pp_g$.
In a nutshell, we must check the hypotheses of the lemma below, which combines familiar arguments from real analysis. We state it so as it applies directly to $F$. Let $\cD$ be $\Oo_I$, the closed positive orthant of $\RR^n$, and $\Pp$ be $\Pp_F$.
Denote by $\interior X$ the interior of a set $X$.

\begin{lemma} \label{lemma:topo}  Let $\widetilde H: \cD \to \RR^n$ be a  function satisfying the following properties.
\begin{enumerate}
\item[(H1)] $\widetilde H$ is a continuous, proper map, i.e., $\lim_{\| v \| \to \infty} \|\widetilde H(v) \| = \infty$.
\item[(H2)] The restriction of $\widetilde H$ to $\interior \cD$ is a $C^1$ map with invertible Jacobians.
\item[(H3)] Some point of $\interior \cD$ is taken by $\widetilde H$ to $\interior \Pp$.
\item[(H4)] The restriction $\hat H: \partial \cD \to \partial \Pp$ is a homeomorphism. 
\item[(H5)] No point of $\interior \cD$ is taken by $\widetilde H$ to $\partial \Pp$.

\end{enumerate}
Then the image of  $\widetilde H$ is $\Pp$ and the function $\widetilde H: \cD \to \Pp$  obtained from $\widetilde H$ by restricting its counterdomain is a homeomorphism which restricts to a diffeomorphism between $\interior \cD$ and $\interior \Pp$.
\end{lemma}

In order to apply the lemma for the function $G$, consider $\widetilde H: \cD \to \RR^{n+1}$, where $\cD = \Oo_I \times \RR$ and set $\Pp = \Pp_G$. The proof follows verbatim.

\medskip

\begin{proof} We first show that points in the connected components of $\RR^n \setminus \widetilde H(\partial \cD ) = \RR^n \setminus \partial \Pp$ have the same number of preimages. Take $\mu \in \RR^n \setminus \partial \Pp$. By connectivity,
it suffices to show that, for a small open neighborhood $U$ of $\mu$, points
in $U$ have the same number of preimages. If $\mu$ has infinite
preimages, by properness (hypothesis (H1)) they have to accumulate at some preimage $v_\ast$.
Preimages in $\interior \cD $ are isolated, by the inverse function
theorem (use hypothesis (H2)), thus $v_\ast \in \partial \cD $, contradicting $\widetilde H(v_\ast) = \mu \in \RR^n \setminus \partial \Pp$.

Thus $\mu$ has a finite number of isolated preimages, say $v_1, \ldots, v_k$. From the inverse function theorem, for every sufficiently small open ball $B$ centered around $\mu$,
	there are open disjoint sets $V_i, i=1, \ldots, k$, each containing $v_i$,
	for which $\widetilde H$ takes $ V_i$ to $B$ diffeomorphically. Thus, points in $B$
	have at least $k$ preimages.

If, for  balls $B_n$ of radius $1/n$ there are
	points $\tilde \mu_n$ with (at least) $k+1$ preimages, one preimage $w_n
	\in \cD $ is outside  $\cup_{i=1}^k V_i$. By properness, they
	accumulate at $w_\ast \notin \cup_{i=1}^k V_i$. But then $\widetilde H(w_\ast)
	= \mu$, contradicting the fact that $\mu$ has exactly $k$ preimages, all
	in $\cup_{i=1}^k V_i$.

By (H3), there is $v \in \interior \cD$ such that $\widetilde H(v) \in \interior \Pp$. From the argument above and $(H4)$, the set $\Pp$ lies in the image of $\widetilde H$.  If, for some $w \in \interior \cD$ we have $\widetilde H(w) \notin \Pp$, then the segment joining $v$ in $w$, which lies in $\interior \cD$, must contain a point whose image lies in $\partial \Pp$, contradicting (H5). Thus, the image of $\widetilde H$ is  $\Pp$ and the associated function $H : \cD \to \Pp$ is well defined.

From (H1), $\widetilde H: \cD \to \RR^n$ is  proper and has a well defined topological degree $\deg(\widetilde H,\mu)$ (an excellent reference for degree theory is \cite{OutereloRuiz}) for any regular value $\mu \in \interior \RR^n \setminus \partial \Pp$, i.e., a point whose preimages are regular points  (regular values are dense, by Sard's theorem). From (H4), for $\mu \in \interior \Pp$, $\deg (\widetilde H, \mu) = \pm 1$. Moreover,
\[ \deg(\widetilde H, \mu) \ = \ \sum_{v \in \widetilde H^{-1}(\mu)} \sgn \det D\widetilde H(v)  \]
and all preimages are counted with the same sign, by (H5). Indeed, the  determinant of the Jacobian $DH(v)$ is never zero for $v \in \interior \cD$ and is continuous, by (H1). Thus every point of $\Pp$ has a unique preimage.
\qed
\end{proof}

\medskip
We are left with proving the hypotheses of the lemma for the counterparts $\widetilde F: \cD_F \to \RR^n$ and $\widetilde G: \cD_G \to \RR^{n+1}$ of the functions $F$ and $G$.

\section{Proof of Theorem \ref{teorema} for $F$ and $F^r$}

\medskip
Without loss,  suppose $S = D$, a diagonal matrix
with eigenvalues
\[ D_{11} = \lambda_1 < \ldots < D_{nn}=\lambda_n \ . \]
We then take $\Qq = [e_1, \ldots, e_n]$ to consist of the canonical vectors, so that $Q = I$ and $\Oo_I \subset \RR^n$ is the usual positive orthant. Consider
\begin{align*} \tilde F : \cD &= \cD_F = \Oo_I \to \ \RR^n  \hspace{0.3cm} ,\\
	v \  &\mapsto \sigma_o (D + v \otimes v)
\end{align*}
where now all numbers in sight are real. Complex numbers will return only in the proof of Theorem \ref{teoremadois} in Section \ref{provateoremadois}.

The set $\partial \cD $ consists of $n$ faces of  $\cD = \Oo_I$,
\[ E_i \ = \ \{ v \in \RR^n \ , \ v_i =0\}, \quad i = 1, \ldots,n \ .\]

The parallelotope $\Pp  \subset \RR^{n}$ has $2 n  - 1$ faces, which we now describe.
Set
\[ L_{i} \ = \ [\lambda_1, \lambda_2] \times \ldots \times [\lambda_{i-1},
\lambda_{i}] \ , \quad R_i \ = \  [\lambda_{i}, \lambda_{i+1}] \times \ldots \times
[\lambda_n, \infty)\ , \]
where sets using indices not in $\{1, \ldots, n\}$ are omitted.
For $i>1$, as we shall see, $\tilde F$ (and $F$) {\it creases}  each face $E_i$, sending it  to two adjoining faces of $\partial \Pp $,
\[  F(E_i) \ = \  \big(L_{i} \times
\{ \lambda_i \} \times R_{i+1} \big) \ \cup \ \big( L_{i-1} \times \{ \lambda_{i} \} \times R_{i} \big)  \ . \]
Face $E_1$ is sent to a single face of $\partial \Pp$, $\{\lambda_1\} \times R_{2}$. The reader is invited to check that the formulas indeed describe five of the six faces of a parallelotope in $\RR^3$.

\medskip
Recall that a simple eigenvalue  of a symmetric matrix varies smoothly with the matrix  \cite{Lax}:  in this case, if
$T w_i \ = \ \lambda_i w_i$ for a normalized $w_i \in \RR^n$,
$ \dot \lambda_i  \ = \ \langle \dot T w_i , w_i \rangle $.

\medskip
We define three subsets of $\cD$.

\begin{enumerate}
\item[.] $\cD_d $ is the  set of points in which $\tilde F(v)$ has a double eigenvalue.
\item[.]  The {\it critical set} $\ \cC \subset \interior (\cD \setminus \cD_d)$ consists of points in  which the Jacobian $D \tilde F$ is not invertible.
\item[.] The set of {\it regular} points is the complement $\cD \setminus (\partial \cD \cup \cD_d \cup \cC)$.
\end{enumerate}

\begin{prop} \label{prop:fronteira}
	\begin{itemize}
		
		\item[(i)] $\cD_d \subset \partial \cD $. Thus, $\tilde F$ is differentiable in
		$\interior \cD $.
		\item[(ii)] $\cC \ = \ \emptyset$.
\item[(iii)]  $\partial \tilde F(\cD ) \subset \tilde F(\partial \cD )$.
		\item[(iv)] The matrices $D$ and $D + v \otimes v$ share an eigenvalue
		$\lambda_i$ if and only if $v \in E_i$. In particular, $\widetilde F^{-1} (\partial \Pp) \subset \partial \cD$.
	\end{itemize}
	\end{prop}

\medskip

\medskip

\begin{proof} 	We prove (i). A double eigenvalue $\lambda_i$ of $D + v \otimes v$ admits a
	(nonzero) eigenvector $w$ in the subspace of eigenvectors associated with $\lambda_i$ for which $w_1 =0$.  In the expression
	$(D-\lambda_i)w \ = \ -(v \otimes v) w \ = \ -\langle v, w \rangle v$, equate
	first coordinates:
	either $v_1=0$ or $\langle v, w \rangle =0$. In the first case, $v \in E_1
	\subset \partial \cD $ and we are done. Otherwise, $(D - \lambda_i) w = 0$
	and $w$ is a canonical vector, $w= e_j$. As $\langle v, w \rangle
	=0$, we
	must have $v_j=0$ and then $v \in E_j \subset \partial \cD$.
	
	\medskip
	
	To prove (ii), let $T \ = \ D + v \otimes v$, $T w_i = \lambda_i w_i, \ i=1,
	\ldots, n, \ \| w_i \| \ = \ 1$.
	The Jacobian of $\tilde F$ at a point $v$ is
	\[ J(v) \dot v \ = \ \big( \langle \dot T w_1,  w_1 \rangle, \ldots, \langle
	\dot T w_{n},  w_{n} \rangle \big) , \]
	where  $\dot v \in \RR^n$ and
	$ \dot T \ = \ \dot v \otimes v + v \otimes \dot v $.
	Let $\dot V$ be the vector space of such matrices.
	Write the linear transformation $J(v)$ as a composition,
	\[ J(v) \dot v \ = \ 2 \big(\langle w_1, v \rangle \langle w_1, \dot v
	\rangle, \ldots,\langle w_n, v \rangle \langle w_n, \dot v \rangle  \big) \]
	\[\ = \ 2 \diag(\langle w_1, v \rangle, \ldots, \langle w_n, v \rangle)\
	(\langle w_1, \dot v \rangle, \ldots, \langle w_n, \dot v \rangle  )^T \ .\]
	
	A point $v$ in the interior of $\cD $ is critical if and only if $J(v)$ not
	invertible. Clearly $ \dot v \mapsto (\langle w_1, \dot v \rangle, \ldots,
	\langle w_n, \dot v \rangle  )$
	is invertible, as the vectors $\{ w_i\}$ are linearly independent. Suppose
	by contradiction that, for some $i$, we have $\langle w_i, v \rangle \ = \ 0$.
	Equation $(D +  v \otimes v) w_i = \lambda_i w_i$ becomes $(D - \lambda_i)
	w_i \ = \ 0$, so that
	$w_i \ = \  t e_i, t \ne 0$. Now, $\langle v, w_i \rangle=0$ implies $v_i =
	0$, and again $v \in E_i \subset \partial \cD $.
	
	\medskip

By the inverse function theorem, $\tilde F$ is a local diffeomorphism at regular points: this settles (iii).
	
	To prove (iv), take a common eigenvalue $\lambda_i$ and eigenvectors $e_i, y
	\ne 0$, so that $D e_i = \lambda_i e_i$ and $D y + \langle v , y \rangle v =
	\lambda_i y$ and
	$ (D - \lambda_i) (y-e_i) \ = \ - \langle v , y \rangle v $. The $i$-th entry of
	both
	sides of the last equation is zero. If $v_i \ = \ 0$, we are done. Suppose
	$\langle v , y \rangle =0$: $y=  t e_i \in E_i, t \ne 0$, and then
	$v_i = 0$. The converse is trivial.
	\qed
\end{proof}

\medskip

For $v \in E_i$, $\lambda_i$ is in the spectrum of both $D$ and $T = D + v \otimes v$, but we do not know yet where $\lambda_i$ sits among the ordered eigenvalues of $T$. As we shall see, the study of $\partial
\tilde F( E_i)$ requires the understanding of the map $\tilde F_{\hat D}$ for
the $(n-1) \times (n-1)$ matrix $\hat D$, obtained from $D$ by removing the
eigenvalue $\lambda_i$. Said differently, the proof that  $\tilde F$ takes
$\partial \cD $ to $\partial \Pp $ homeomorphically is by induction.

\medskip
Statements (i) and (ii) above imply hypothesis (H2) of Lemma \ref{lemma:topo}.

We verify hypotheses (H1) and (H3).

\begin{prop} \label{prop:properF2}
	Hypothesis (H1) of Lemma \ref{lemma:topo} holds: the map $\tilde F$ is proper.
\end{prop}

\begin{proof} With the Frobenius norm, for a matrix $T = D + v \otimes v$ with
eigenvalues $\{\mu_i\}$, we have $\|D + v \otimes v \|^2 \ = \ \sum_i \mu_i^2$.
	\qed
\end{proof}

\begin{prop} \label{prop:H3} Hypothesis (H3) also holds:
for some $v \in \interior \cD$, $\widetilde F(v) \in \interior \Pp$.
	
\end{prop}

\begin{proof} Let $\bum = (1, 1, \ldots, 1)$, $t>0$. For eigenvalues $\lambda_i$ of $D + t v \otimes v$, we have $\dot \lambda_i (t=0) = \langle \bum \otimes \bum \ e_i, e_i \rangle = 1 >0$: $\sigma_o(D + t \ \bum \otimes \bum)$ enters $\Pp$ for $t>0$ small.
\qed
\end{proof}

\medskip
Hypotheses (H4) and (H5) require an inductive argument, presented below.

\bs
\noi
{\bf Proof of Theorem \ref{teorema} for $F$ and $F^r$:} We first prove by induction the claim about $F$, and then we handle $F^r$. The case $n \ = \ 2$ contains the gist of the proof.
For $v \in E_1$, $v = (0, c)$, so that
\[ T \ = \ D + v \otimes v \ = \ \begin{pmatrix} \lambda_1 & 0 \\ 0 & \lambda_2 + c^2
\end{pmatrix} \ . \]
As $\lambda_1 < \lambda_2$, we also have $\lambda_1 < \lambda_2 + c^2$, so that
\[ \tilde F(T) = (\lambda_1 , \lambda_2 + c^2) \in \{ \lambda_1\} \times
[\lambda_2, \infty) \in \partial \Pp \ . \]
If instead $v \in E_2$, $v = (c, 0)$ and
\[ T \ = \ D + v \otimes v \ = \ \begin{pmatrix} \lambda_1 + c^2 & 0 \\ 0 & \lambda_2
\end{pmatrix} \ . \]
There are two possibilities. If $\lambda_1 + c^2 \le \lambda_2$, in accordance with Figure \ref{figura},
\[ \tilde F(T) \ = \ (\lambda_1 + c^2 , \lambda_2) \in [\lambda_1, \lambda_2]
\times \{ \lambda_2\} \in \partial \Pp \ .\]
Otherwise
\[ \tilde F(T) \ = \ (\lambda_2, \lambda_1 + c^2 ) \in \{\lambda_2\} \times [
\lambda_2, \infty) \in \partial \Pp  \ .\]
As $\tilde F: \cD  \to \RR^n$ is proper,  the restriction  $F:\partial \cD  \to
\partial \Pp $ is a homeomorphism: $F$ satisfies hypothesis (H4) of Lemma \ref{lemma:topo}. From Lemma \ref{prop:fronteira}(iv), (H5) also holds. The first  step of the induction argument is complete.

\medskip

We assume the claim for $\widetilde F$ acting on $(n-1) \times (n-1)$ matrices.
For a diagonal $n \times n$ matrix $D$, we consider $\tilde F(\partial \cD )$.
For  $v \in E_i$ in a face of the orthant  $\cD$, the
$i$-th column and row of the  matrix $D + v \otimes v$ equal $\lambda_i e_i^T$
and $\lambda_i e_i$, so that $\lambda_i$ is a common eigenvalue of $D$ and $D + v \otimes v$. The remaining
eigenvalues of $D + v \otimes v$  belong to the spectrum of $\hat D + \hat v \otimes \hat v$, where
$\hat D$ is obtained by removing the $i$-th row and column of $D$ and $\hat v$ is obtained from removing the $i$-th entry of $v$.

In order to apply the inductive
hypothesis, at the risk of being pedantic, identify $E_i$ with the positive
orthant $\hat \cD  \subset \RR^{n-1}$, so that $\tilde F: E_i \to \RR^{n-1}$ is
 identified with  $\tilde F_{\hat D}: \hat \cD  \to \RR^{n-1}$ which, by induction, induces a homeomorphism  $F_{\hat D}: \hat \cD  \to \hat \Pp $, where, for $i > 1$,
\[ \hat \Pp  \ = \  [\lambda_1,\lambda_2] \times \ldots \times
[\lambda_{i-2},\lambda_{i-1}] \times [\lambda_{i-1},\lambda_{i+1}]\times
[\lambda_{i+1},\lambda_{i+2}] \times \ldots \times [\lambda_n, \infty). \]
Notice that the two intervals containing $\lambda_i$ in the definition of $\Pp $
were replaced by a single interval $[\lambda_{i-1},\lambda_{i+1}]$. Split $[\lambda_{i-1},\lambda_{i+1}] = [\lambda_{i-1},\lambda_{i}] \cup [\lambda_{i},\lambda_{i+1}]$, and then
\[ \hat \Pp  \ = \  \big( [\lambda_1,\lambda_2] \times \ldots \times
[\lambda_{i-2},\lambda_{i-1}] \times [\lambda_{i-1},\lambda_{i}]\times
[\lambda_{i+1},\lambda_{i+2}] \times \ldots \times [\lambda_n, \infty) \big)\]
\[\cup \ \big([\lambda_1,\lambda_2] \times \ldots \times
[\lambda_{i-2},\lambda_{i-1}] \times [\lambda_{i},\lambda_{i+1}]\times
[\lambda_{i+1},\lambda_{i+2}] \times \ldots \times [\lambda_n, \infty) \big) . \]
In order to compute $\tilde F(E_i)$, we insert $\lambda_i$
among the ordered eigenvalues in $\hat \Pp$.
\[ F(E_i) \ = \ \big( [\lambda_1,\lambda_2] \times \ldots \times
[\lambda_{i-2},\lambda_{i-1}] \times [\lambda_{i-1},\lambda_{i}]\times \{ \lambda_i\} \times
[\lambda_{i+1},\lambda_{i+2}] \times \ldots \times [\lambda_n, \infty) \big)\]
\[\cup \ \big([\lambda_1,\lambda_2] \times \ldots \times
[\lambda_{i-2},\lambda_{i-1}] \times \{ \lambda_i\} \times [\lambda_{i},\lambda_{i+1}]\times
[\lambda_{i+1},\lambda_{i+2}] \times \ldots \times [\lambda_n, \infty) \big) \]
\[ \ = \  \big(L_{i} \times
\{ \lambda_i \} \times R_{i+1} \big) \ \cup \ \big( L_{i-1} \times \{ \lambda_{i} \} \times R_{i} \big)  \]
in the notation introduced in the beginning of the section. Thus $F$ indeed creases  faces $E_i$, $i>1$, giving rise to two faces of $\partial \Pp$. Moreover, $F$ is a homeomorphism between
the remaining faces $E_1 \subset \partial \cD$ and $\big(\{ \lambda_1 \} \times R_{2} \big) \subset \partial \Pp$: the details are left to the reader (simply omit intervals containing the index $i-1$).

Thus, $F:\partial \cD  \to \partial \Pp $ is surjective, and injective  on the
restriction to each face $E_i \times \RR$. We are left with showing injectivity
on the union of the faces. Let $v_i \in E_i$ and $v \in \cD $ such that $\tilde
F(v_i)= \tilde F(v)$. As $\lambda_i$ is an eigenvalue of $\tilde F(v_i)$, by
Proposition \ref{prop:fronteira}(iv) we must have $v \in E_i$. As the
restriction of $\tilde F$ to $E_i$ is injective, global injectivity in
$\partial \cD $ follows. A simple argument then shows that $F: \partial \cD  \to \partial \Pp $ is a
homeomorphism, so that hypothesis (H4) of Lemma \ref{lemma:topo} holds. From Lemma \ref{prop:fronteira} (iv), (H5) also holds.  Item (1) now follows from Lemma \ref{lemma:topo}.

\medskip
We now consider $F^r$. Again, $S = D$. Since
\[ \tr (S + v \otimes v) \ = \ \tr D + \langle v, v \rangle ^2\]
and $F: \cD  \to \Pp $ is a homeomorphism, we have that $F^r$ is also a
homeomorphism. When restricting to the interior of $\cD$, $F^r$ takes one
hypersurface to another and the Jacobian at each point is easily seen to be
invertible, showing that  $F^r$ is indeed a diffeomorphism between interiors.
\qed

\section{Proof of Theorem \ref{teorema}  for $G$, $G^c$ and $G^{r,c}$}

Again, without loss, $S \ = \ D$, a diagonal matrix with  eigenvalues $\lambda_1 < \ldots < \lambda_n$.
Now $\cD  \ = \ \Oo_I \times \RR$ has faces of the form $E_i \times \RR$ and the
box
\[ \Pp = \Pp_G = (-\infty, \lambda_1] \times [\lambda_1,\lambda_2] \times
[\lambda_2,\lambda_3] \times \ldots \times [\lambda_n, \infty)  \subset
\RR^{n+1} \]
has $2n$ faces. Recall that all numbers in sight are real.
Define
\[
T = T(v,c)  \ = \ \begin{pmatrix}
	D & v \\
	v^\ast & c
\end{pmatrix} \ .
\]

We must show that the map
$ \tilde G : \cD  \to \RR^{n+1} , \
( v , c ) \mapsto \sigma_o (T(v,c)) = \sigma_o (T)$
defines a homeomorphism $G: \cD  \to \partial \Pp$. This time, as we shall see,
$G$ takes every face of $\partial \cD $ to two adjoining faces of $\partial
\Pp$.

As before $\cD_d$ consists of the points $(v,c) \in \cD $ for which $T(v,c)$ has
a double eigenvalue,  the critical set $\cC$ is the set of points in the
interior of $\cD \setminus \cD_d $ in which $G$ is differentiable with not invertible Jacobian,
and its complement in $\cD $ is the set of regular points. The counterpart of
Proposition \ref{prop:fronteira} still holds.

\begin{proposition} \label{prop:fronteirabord}
	 (i) $\cD_d \subset \partial \cD $, so that $\widetilde G$ is differentiable in $\interior \cD$,  (ii)  $\cC \ = \ \emptyset$, (iii) $\partial \tilde G(\cD ) \subset \tilde G(\partial \cD)$,
	(iv)
	The matrices $D$ and $T(v,c)$ share an eigenvalue $\lambda_i$ if and only
	if $v \in E_i$ and $(v,c) \in \partial \cD $. Thus $\tilde G^{-1}(\partial \cD ) \subset \partial \cD$.
\end{proposition}

\medskip
\begin{proof}  For (i), take a double
	eigenvalue $\rho$ and an associated eigenvector $w \in \RR^{n+1}$
	with $w_{n+1} \ = \ 0$. Expanding $(T(v,c) - \rho) w = 0$ we have that $\rho
	= \lambda_i$ for some $i = 1, \ldots, n$, is an eigenvalue of $D$ and $w = t e_i, t \ne 0$. But then
	$(T(v,c) - \lambda_i) e_i \ = \ 0$ implies that $v_i = 0$, so that $v \in E_i$. For (ii), imitate the argument in the previous section:
	
	\[ DG(v,c) (\dot v, \dot c) \ = \ \big( \langle w_1, \dot T  w_1 \rangle,
	\ldots, \langle w_{n+1}, \dot T w_{n+1} \rangle \big) , \]
	where $T(v,c)\ w_j = \lambda_j\ w_j,\ j=1, \ldots, n+1, \ \| w_j \| \ = \ 1$ and
	\[ \dot T \ = \ {\dot T}(v,c)(\dot v, \dot c) \ = \ \begin{pmatrix} 0 & \dot v
	\\
		\dot v^\ast & \dot c\end{pmatrix} \ .\]
	Let $\dot V\simeq \RR^{n+1}$ be the vector space spanned by the
	matrices $\dot T$.

	Using Frobenius inner products, the Jacobian $DG: \dot V \to \RR^{n+1}$ becomes
	\[ DG(v,c) (\dot v, \dot c)  \ = \ \big(\tr (w_1 \otimes w_1) \dot T, \ldots,
	\tr (w_{n+1} \otimes w_{n+1}) \dot T \big)\]
	\[\ = \ \big(\langle \dot T , w_1 \otimes w_1 \rangle, \ldots, \langle \dot T , w_{n+1}
	\otimes w_{n+1}\rangle \big) . \]
	Thus, a point $(v,c) \in \interior \cD $ is critical if and only if every
	linear combination of the eigenprojections $w_k \otimes w_k$ is orthogonal to some nonzero matrix $\dot T \in \dot V$. As $T$ has simple spectrum (by (i)), such linear combination is a polynomial in $T$,
 \[ \sum_{j=1}^{n+1} c_j\ w_j \otimes w_j \ = \
	\sum_{k=0}^{n} d_k\ T^k \ .\]
Thus there is $\dot T$ orthogonal to all polynomial functions $p(T)$.
	The inner product of $\dot T \in \dot V$ with an arbitrary real
	symmetric matrix $M$  is simply
	\begin{equation*}
		\langle \dot T , \begin{pmatrix} * & y \\
			y^\ast & x \end{pmatrix} \rangle \ = \
\langle \begin{pmatrix} 0 & \dot v
			\\
			 \dot v^\ast &  \dot c \end{pmatrix} , \begin{pmatrix} * & y \\
			y^\ast & x \end{pmatrix} \rangle \ = \
\langle (2\dot v, \dot c), (y,x) \rangle \ = \  \langle (2\dot v, \dot c ), M e_{n+1} \rangle \ ,
	\end{equation*}
	where $e_{n+1} \ = \ (0, \ldots, 1) \in \RR^{n+1}$ is canonical.
	 Thus, a point $(v,c)$ corresponding to a
	matrix $T \ = \ T(v,c)$ is critical if and only if there is a matrix $\dot T$
	associated with a nonzero $(\dot v, \dot c)$ such that $e_{n+1}, Te_{n+1},
	\ldots, T^n e_{n+1}$ are orthogonal to $(\dot v, \dot c)$. 
	
	Thus $T$ is critical if and only if the vectors $e_{n+1}, Te_{n+1}, \ldots, T^n
	e_{n+1}$ are linearly dependent, i.e., $e_{n+1}$ is not a cyclic vector of
	$T$. Diagonalize $T \ = \ Q_T^T D_T Q_T$, where the rows of the orthogonal
	matrix $Q_T$ are the eigenvectors of $T$ and $D_T$ has simple spectrum,
	from (i). The vectors
	$e_{n+1}, Te_{n+1}, \ldots, T^n e_{n+1}$ are linearly dependent if and only
	if
	the vectors $Q_T e_{n+1}, D_T Q_T e_{n+1}, \ldots, D_T^n Q_T e_{n+1}$ are.
Let $M$ be the matrix having such vectors as columns and define
\[ q = Q_T e_{n+1} \ , \quad D_T= \diag (d_1, d_2, \ldots, d_{n+1}) \ . \]
 Then
\[ M = \begin{pmatrix}
q_1 & 0 & 0 & \ldots & 0 \\
0 & q_2 & 0 & \ldots & 0 \\
 &  & \ldots &  &  \\
0 & 0 & 0 & \ldots & q_{n+1} \\ \end{pmatrix}  \begin{pmatrix}
1 & d_1 & d_1^2 & \ldots & d_1^{n} \\
1 & d_2 & d_2^2 & \ldots & d_2^{n} \\
 &  & \ldots &  &  \\
1 & d_{n+1} & d_{n+1}^2 & \ldots & d_{n+1}^{n+1} \\ \end{pmatrix}  \]
	Since the $d_i$'s are distinct, the (Vandermonde) determinant of the matrix on the right is nonzero and $\det M$ is zero if and only if
	some coordinate $q_i$ of $Q^T e_{n+1}$ is.
	Said differently, the last coordinate of some eigenvector of $T$ is zero. Say $w
	\ = \
	( \tilde w,0)$ satisfies $(T - \lambda_i) w = 0$. Then $(D - \lambda_i)
	\tilde
	w = 0$, so that $\lambda_i$ is also an eigenvalue of $D$. Since $D$ has
	simple
	spectrum, we must have $\tilde w = \alpha e_k$ for some $\alpha \ne 0$ and
	$e_k
	\in \RR^n$ a canonical vector, Thus, without loss, $w \ = \ (e_k,0)$. Equating the $(n+1)$-th entry $(n+1)$ of $(T - \lambda_i) w = 0$, we obtain $T_{k, n+1} =
	T(v,c)_{k, n+1} = v_k \ = \ 0$. Thus $v_k \in E_k$ and $(v,c) \in \partial
	\cD $.
	The proof of (ii) is complete.

Item (iii) again follows from (ii) and the inverse function theorem.
	
	To prove (iv), simply expand $\det( T(v, c) - \lambda_i I)$ along row $i$.
	\qed
\end{proof}

\medskip
Again, statements (i) and (ii) above imply hypothesis (H2) of Lemma \ref{lemma:topo}. Hypothesis (H1) is proved mimicking Proposition  \ref{prop:properF2}, but (H3) is more delicate.

\begin{prop} \label{prop:H3G} (H3) holds for $G$:
for some $v \in \interior \cD$, $G(v) \in \interior \Pp$.
\end{prop}

\begin{proof} Fix $\lambda_{n+1} = c > \lambda_n$ and consider $\bum = (1,1, \ldots, 1) \in \RR^n$,   $t >0$,
\[
T = T(t)  \ = \ \begin{pmatrix}
	D & 0 \\
	0 & \lambda_{n+1}
\end{pmatrix} + t \begin{pmatrix}
	0 & \bum \\
	\bum^\ast & 0
\end{pmatrix} \ .
\]
 For eigenvalues $\lambda_j(t)$, $j=1, \ldots, n+1$, of $T(t)$,
 \[ \dot \lambda_j (t=0) = \langle \begin{pmatrix}
	0 & \bum \\
	\bum^\ast & 0
\end{pmatrix} \ e_j, e_j \rangle = 0 \]
and we must compute second derivatives. Define normalized eigenvectors $w_j(t)$ such that $T(t) w_j(t) = \lambda_j(t) w_j(t)$, $w_j(0) = e_j$. Then (\cite{Lax})
\[  \dot \lambda_j (t) = \langle \begin{pmatrix}
	0 & \bum \\
	\bum^\ast & 0
\end{pmatrix} \ w_j(t), w_j(t) \rangle  \ , \quad \dot w_j(t) \ = - (T(t) - \lambda_j(t))^{-1} (\dot T(t) - \dot \lambda_j(t)) w_j(t) \ , \]
where the map being inverted is the restriction
$  T(t) - \lambda_j(t): \{ w_i\}^\perp \to \{ w_j\}^\perp$.
For $t = 0$, as $\dot \lambda_j(0)=0$ for $j=1, \ldots, n+1$,
\[ \dot w_j(0) \ = - \begin{pmatrix}
	D  - \lambda_j& 0 \\
	0 & \lambda_{n+1} - \lambda_j
\end{pmatrix}^{-1} \begin{pmatrix}
	0 & \bum \\
	\bum^\ast & 0
\end{pmatrix}  e_j  \]
so that, for $j=1, \ldots, n$ and $j= n+1$ we have, respectively,
\[ \dot w_j (0) = - \frac{1}{\lambda_{n+1}- \lambda_j} \ e_{n+1} \ \ \hbox{or} \ \ \dot w_{n+1}(0) = - \begin{pmatrix}
	(D - \lambda_{n+1})^{-1} \bum\\
	 0
\end{pmatrix} \ . \]
We are ready to compute the second derivative of $\lambda_j$ at $t=0$,
\[ \ddot \lambda_i (0) \ = \ 2  \langle \begin{pmatrix}
	0 & \bum \\
	\bum^\ast & 0
\end{pmatrix} \ \dot w_i(0), w_j(0) \rangle \ .\]
For $j=1, \ldots, n$, as $\lambda_{n+1} >  \lambda_i$,
\[ \ddot \lambda_j (0) \ = \ -2  \langle \begin{pmatrix}
	0 & \bum \\
	\bum^\ast & 0
\end{pmatrix} \  \frac{1}{\lambda_{n+1}- \lambda_j} \ e_{n+1}, e_j \rangle  \ = \
 \frac{-2}{\lambda_{n+1}- \lambda_i} < 0 \ . \]
 For $j = n+1$,
\[ \ddot \lambda_{n+1} (0) = -2  \ \langle \begin{pmatrix}
	0 & \bum \\
	\bum^\ast & 0
\end{pmatrix} \   \begin{pmatrix}
	(D - \lambda_{n+1})^{-1} \bum\\
	 0
\end{pmatrix}, e_{n+1} \rangle  = (-2) \langle \bum, (D - \lambda_{n+1})^{-1} \bum \rangle > 0 \]
and  $\sigma_o(T(t))$ indeed belongs to $\interior \Pp$ for small $t > 0 $.
\qed
\end{proof}

\medskip
\noi{\bf Proof of Theorem \ref{teorema} for $G$, $G^c$ and $G^{r,c}$:}
We first use  induction to prove the claim for $G$. For $n \ = \ 1$, $\cD  \ = \ [0, \infty) \times \RR$, so
that  $\partial \cD = \{(0, c), c \in \RR\}$. The
eigenvalues of $T(0,c)$ are $\{\lambda_1, c\}$ and must be ordered. If $c
< \lambda_1$ then
\[ G(0,c) = (c, \lambda_1) \in (-\infty, \lambda_1] \times \{\lambda_1\} \ . \]
If $c >\lambda_1$ then $G(0,c) = (\lambda_1, c) \in \{\lambda_1\} \times
[\lambda_1, \infty)$.
If $c \ = \ \lambda_1$, $G(0,c)$ lies in the common subface, a single point of
double spectrum associated with the diagonal matrix $\lambda_1 I$.  Again, it
is the ordering which creases $\partial \cD $, a straight line, so as to cover
both faces of $\Pp$. The first inductive step is complete.

Take a diagonal $n \times n$ matrix $D$: we consider $G(\partial \cD )$. The
$i$-th face of $\partial \cD$ is $E_i  \times \RR$. Let
\[ L_{i} \ = \ (-\infty, \lambda_1] \times \ldots \times [\lambda_{i-1},
\lambda_i] \ , \ R_i \ = \  [\lambda_{i}, \lambda_{i+1}] \times \ldots \times
[\lambda_n, \infty) \ , \] where again sets with indices not in $\{1, \ldots, n\}$ are omitted in formulas.
Following the argument in the previous section,  for each face $E_i \times \RR, \ i=1, \ldots, n$,
\[  G(E_i) \ = \  \big(L_{i} \times
\{ \lambda_i \} \times R_{i+1} \big) \ \cup \ \big( L_{i-1} \times \{ \lambda_{i} \} \times R_{i} \big) \ . \]
Thus, $G:\partial \cD  \to \partial \Pp$ is surjective, and restricts
injectively to each face $E_i \times \RR$. We are left with showing injectivity
on the union of the faces.

For $(v,c) \in E_i \times \RR$ set $M = T(v,c)$ and consider another bordered
matrix $N = T(w, d)$ with the same ordered spectrum. As $\lambda_i$ belongs to
$\lambda(M)$, we must have $\lambda_i \in \lambda(N)$ and, from the previous
lemma, $w_i =0$: $(w,v) \in E_i \times \RR$. Global injectivity now follows
from injectivity of $G$ restricted to $E_i \times \RR$.

The claim about $G^c$ follows the proof of the claim about $F^r$: compare traces. For $G^{r,c}$,  compare traces of squares,
\[ \sum_j \mu_j^2 = \tr T(v,c)^2 = \tr D^2 + 2 \langle v, v \rangle + c^2  = \sum_k \lambda_k^2 + 2 r ^2 + c^2  \]
and imitate the rest of the argument relating two surfaces of codimension 2.
\qed

\section{Proof of Theorem \ref{teoremadois}} \label{provateoremadois}

We consider $\hat F : \KK^n \to \Pp_F$, the other cases being analogous.

\bigskip

From Theorem \ref{teorema}, as $F$ is injective, $F(\abs(v)) = F(\abs(w))$ if and only if $v = w$. From the surjectivity of $F$, given $\mu \in \Pp_F$, there is a (unique) $v \in \Oo_Q$ for which $F(v) = \mu$. Hence,  $\hat F^{-1}(\mu) = \abs^{-1} \circ F^{-1}(\mu)=\abs^{-1}(v)$.

Each nonzero coordinate $v_k$ of $v$ gives rise to a circle $e^{i \theta_k} v_k$ of possible values for the $k$-th coordinate of $\abs^{-1}(v)$. Clearly $z_k = 0$ if and only if $z \in E_k$. \qed

\bibliographystyle{siamplain}

\end{document}